\documentclass[12pt]{article}% Math and Physical Sciences Reference Style

\usepackage{authblk} % \affil
\usepackage{hyperref} % ref in PDF
\usepackage{dsfont}
\usepackage[T1]{fontenc}
\usepackage{amsmath}
\usepackage{amsfonts,amssymb}
\usepackage{extarrows}
\usepackage{geometry}
\usepackage{enumerate}
\usepackage{enumitem}

\usepackage[utf8]{inputenc}
\def\0{\emptyset}

\newtheorem{theorem}{Theorem}[section]
\newtheorem{problem}{Problem}[section]

\newtheorem{claim}[theorem]{Claim}

\newtheorem{conjecture}[theorem]{Conjecture}

\newenvironment{proof}{{\noindent\it Proof.}}{\hfill $\square$\par}
\usepackage{graphicx}

\usepackage{enumerate}

\usepackage{
	amsmath,			% Math Environments
	amssymb,			% Extended Symbols
	enumerate,		    % Enumerate Environments
	graphicx,			% Include Images
	lastpage,			% R\part{title}eference Lastpage
	multicol,			% Use Multi-columns
	multirow,			% Use Multi-rows
	pifont,			    % For Checkmarks
}

\usepackage[numbers]{natbib}
\begin{document}

% --- PAPER INFO ---

\title{A generalization of Erd\H{o}s-Hajnal problem on paths with equal-degree endpoints}

\author[1]{\small\bf Xiamiao Zhao\thanks{E-mail: zxm23@mails.tsinghua.edu.cn}}
\author[1]{\small\bf Yichen Wang\thanks{\textit{corresponding author}: E-mail: wangyich22@mails.tsinghua.edu.cn}}
% \author[2]{\small\bf Ervin Gy\H{o}ri\thanks{E-mail:  gyori.ervin@renyi.hu}}
\author[1]{\small\bf Mei Lu\thanks{E-mail: lumei@tsinghua.edu.cn}}

\affil[1]{\small Department of Mathematical Sciences, Tsinghua University, Beijing, P.R. China.}

%\author{Yichen Wang}
\date{}

\maketitle

\begin{abstract}
Erd\H{o}s and Hajnal proposed a problem that: is it true that every $(2n+1)$-vertex graph with $n^2+n+1$ edges contains two vertices of equal degree connected by a path of length three?
The edge bound is sharp by the complete bipartite graph $K_{n,n+1}$.
Recently, Chen and Ma [Journal of Combinatorial Theory, Series B, 179:1-18, 2026] answered this problem affirmatively for every $n \ge 600$.
In the same paper, they further conjectured that for sufficiently large $n$, the statement is true if we replace the path of length three by a path of fixed odd length.
In this paper, we confirm their conjecture.
\end{abstract}

%\textbf{AMS classification: }\textit{05C75, 05C65, 05C05}\vskip 0.3cm

\vskip.3cm

% -------------------------------------------------------------
% --------- Here begins INTRODUCTION SECTION ------------------
\section{Introduction}

In 1991, Erd\H{o}s and Hajnal proposed the following problem.

\begin{problem}[Erd\H{o}s and Hajnal~\cite{erdos1991problems}]
Is it true that every $(2n+1)$-vertex graph with $n^2+n+1$ edges contains two vertices of equal degree connected by a path of length three?
\end{problem}

The problem is also listed as Problem \#816 in Thomas Bloom's collection of Erd\H{o}s problems\footnote{\url{https://www.erdosproblems.com/816}}.
Recently, Chen and Ma \cite{CHEN20261} answered this problem affirmatively for every $n \ge 600$.

\begin{theorem}[Chen and Ma \cite{CHEN20261}]
	Let $n \ge 600$. The unique $(2n+1)$-vertex graph with at least $n^2+n$ edges that does not contain two vertices of the same degree joined by a path of length three, is the complete bipartite graph $K_{n,n+1}$.
\end{theorem}

They actually proved a stronger result since the property of having two vertices of equal degree joined by a path of length three is not a monotone graph property.
Moreover, they proved the extremal graph is unique.
The bound $n \ge 600$ was later improved to $n \ge 2$ by Liu and Zeng~\cite{2025arXiv250500523L}.

Chen and Ma also conjectured that the statement is true if we replace the path of length three by a path of odd length.
\begin{conjecture}[Chen and Ma \cite{CHEN20261}]
	Let $\ell \ge 2$ be a fixed integer.
	For sufficiently large $n$, the maximum number of edges in a graph that does not contain two vertices of the same degree joined by a path of length $2\ell+1$, is at most $n^2+n$.
\end{conjecture}

Recently, Liu and Zeng~\cite{2026arXiv260411664L} confirmed the conjecture when $\ell = 2$ for every $n \ge 11$.
In this paper, we confirm their conjecture when $n$ is sufficiently large.
Actually, we have a stronger result, which shows that the $K_{n,n+1}$ is the unique extremal graph.

\begin{theorem}\label{thm: main 2n+1}
	Let $\ell \ge 3$ be a fixed integer.
	Let $G$ be a graph on $2n+1$ vertices with at least $n^2+n$ edges such that no two vertices with the same degree are connected by a path of length $2\ell+1$.
	When $n \ge f(\ell)$ where $f(\ell)$ is a polynomial function of $\ell$, we have that $e(G) = n^2+n$ and $G$ is the complete bipartite graph $K_{n,n+1}$.
\end{theorem}

% -------------------------------------------------------------
% \section{Preliminaries and proof sketch}\label{sec: preliminary}

\section{Preliminaries}\label{sec: preliminaries}

We first introduce some notation.
For a graph $G$, let $e(G)$ be the number of edges in $G$.
For a vertex set $S$, let $G[S]$ be the subgraph of $G$ induced by $S$.
For two disjoint vertex sets $S_1$ and $S_2$, denote by $G[S_1, S_2]$ the bipartite subgraph of $G$ with vertex set $S_1 \cup S_2$ and edges joining vertices in $S_1$ and $S_2$.
%Let $e(S_1,S_2)$ be the number of edges in $G[S_1, S_2]$.
For a vertex $v\in V(G)$, let $N(v)=\{u\in V(G):uv\in E(G)\}$, $N[v]=N(v)\cup\{v\}$, and $d(v)=|N(v)|$ be the degree of $v$.
For  a vertex set $S\subseteq V(G)$, let $N_S(v)=N(v)\cap S$ and $d_S(v)=|N_S(v)|$.
Denote by $P_k$ a path on $k$ vertices.

We follow the ideas of Chen and Ma~\cite{CHEN20261}, and Liu and Zeng~\cite{2025arXiv250500523L}.
Let $G$ be a graph on $2n+1$ vertices with at least $n^2+n$ edges such that no two vertices with the same degree are connected by a $P_{2\ell+2}$.
In this paper, we let $\ell \ge 3$.
Our goal is to prove that $e(G) = n(n+1)$ when $n$ is sufficiently large, and it only holds when $G$ is the complete bipartite graph $K_{n,n+1}$.

Let $\mathcal{M}$ be the set of missing edges in $G$.
Then we have that
\begin{equation}\label{eq: missing edges}
	|\mathcal{M}| = \binom{2n+1}{2} - e(G)\le n^2.
\end{equation}
For convenience, let $\mathcal{M}[S]$ be the set of missing edges in $G[S]$ and $\mathcal{M}[S_1, S_2]$ be the set of missing edges in $G[S_1, S_2]$.

Let $\Delta$ be the maximum degree of  $G$.
Since the degree of each vertex in $G$ lies in $\{0,1,\dots,2n\}$, and  $|V(G)|=2n+1$, by the pigeonhole principle, there exist two vertices with the same degree.
Let $\beta$ be the largest integer such that there exist at least two vertices with degree $\beta$.
Let $S$ be the sum of degrees of all vertices in $G$. Then
\begin{equation}\label{eq: sum of degrees lower bound}
	S = \sum_{v \in V} d(v) = 2e(G)\geq 2n^2+2n.
\end{equation}
We will consider two cases when $\beta > n$ and $\beta \le n$, which are solved in Sections~\ref{sec: beta > n} and \ref{sec: beta < n}, respectively.
From now on, we always assume that $n \ge f(\ell)$ for some large polynomial function $f(\ell)$ of $\ell$.

\section{Proof of Theorem \ref{thm: main 2n+1} in the \texorpdfstring{$\beta > n$ case}{β > n}}\label{sec: beta > n}
Recall that $G$ is a graph on $2n+1$ vertices with at least $n^2+n$ edges, such that no two vertices with the same degree are connected by a $P_{2\ell+2}$, where $P_k$ is the path on $k$ vertices.
We need the Erd\H{o}s-Gallai theorem as follows.

\begin{theorem}[Erd\H{o}s and Gallai~\cite{erdos1959maximal}]\label{thm: Erd\H{o}s and Gallai}
	Let $G$ be a $P_k$-free graph on $n$ vertices, then $e(G) \le \frac{(k-2)n}{2}$.
\end{theorem}

Let $u_1$ and $u_2$ be the two vertices with degree $\beta$ in $G$.
Let $\alpha = \mathbf{1}_{u_1u_2 \in E(G)}$, that is, $\alpha = 1$ if $u_1u_2 \in E(G)$, and $\alpha = 0$ otherwise.
Let $c = \beta - n - \alpha$. Then  we have $c \ge 0$.

Define
\begin{equation*}
\begin{aligned}
	B_{12} &= N(u_1) \cap N(u_2),\\
	B_{1} &= N(u_1) \setminus N[u_2],\\
	B_{2} &= N(u_2) \setminus N[u_1],
\end{aligned}
\end{equation*}
and $B = B_{1} \cup B_{2} \cup B_{12}$, $D = V(G) \setminus (B\cup \{u_1,u_2\})$.

From now on, we always use the corresponding lower case letters to denote the size of each set.
For example, $b_{i} = |B_i|, d = |D|$ and $|B_{12}| = b_{12}$. Note that $N(u_i)\setminus\{u_{i+1}\}=B_i\cup B_{12}$ for $i=1,2$, where the subscript modulo 2. Thus $\beta -\alpha= b_1+b_{12}=b_2+b_{12}$. Moreover, we have
\begin{equation}\label{eq: using c}
\begin{aligned}
	& b = 2(\beta - \alpha) - b_{12} = 2n + 2c - b_{12} = 2n-1-d, \\
	& d = 2n-1 - b = b_{12}-2c-1, \\
	& b_1 = b_2 = \beta - \alpha - b_{12} = c+n-b_{12} = n-c-1-d.
\end{aligned}
\end{equation}
By (\ref{eq: using c}), $b_1=b_2\le n$.
Note that $G[B_{12}]$ is $P_{2\ell}$-free; otherwise, there is a path with length $2\ell+1$ connecting $u_1$ and $u_2$. Thus, $e(G[B_{12}]) \le 2\ell n$ by Theorem~\ref{thm: Erd\H{o}s and Gallai}. Similarly, $G[B_1, B_{12}]$, $G[B_2, B_{12}]$ and $G[B_1, B_2]$ are $P_{2\ell}$-free, and each of them has at most $2\ell n$ edges.
Then, $e(G[B_{12}]) + e(G[B_1,B_{12}]) + e(G[B_2,B_{12}]) + e(G[B_1,B_2]) \le 8\ell n$.
\begin{claim}\label{claim: lower bound of D}
	$d > n^{0.9}$.
\end{claim}

\begin{proof}
	Suppose otherwise $d \le  n^{0.9}$.
	Then we have
	\begin{equation}\label{eq: prove lb of D}
	\begin{aligned}
		n^2 + n \le e(G) &\le 2\beta + e(G[B_1]) + e(G[B_2]) + e(G[B_{12}])
		+ e(G[B_1,B_{12}])\\
		 &+ e(G[B_2,B_{12}]) + e(G[B_1,B_2]) + e(G[D]) + e(G[B,D]) \\
		 &\leq 2\beta + e(G[B_1]) + e(G[B_2]) + 8\ell n + e(G[D]) + e(G[B,D])\\
		&\le 2\beta + \binom{b_1}{2}+\binom{b_2}{2} + 3n^{1.9} - |\mathcal{M}[B_1]| - |\mathcal{M}[B_2]|.
	\end{aligned}
	\end{equation}
The last inequality holds because $e(G[D])+e(G[B,D])+8\ell n\le 2nd+8\ell n\le 3n^{1.9}$ since $d\le n^{0.9}$ and $n$ is large enough.
Note that $b_1=b_2=\beta -\alpha -b_{12}$. Then (\ref{eq: prove lb of D}) implies that $b_1 = b_2 \ge n-n^{0.95}$, otherwise we have a contradiction when $n$ is sufficiently large.
Define
$$\tilde{B}_i = \{v \in B_i \mid d(v) \ge 1.1n\}, i=1,2.$$
%Let $E_1$ be the set of edges with one endpoint in $B_1$ and the other endpoint outside of $B_1$.
Then $e(G[B_1,V\setminus B_{1}]) \le b_1 + e(G[B_1,B_{12}]) + e(G[B_1,B_2]) + e(G[B_1,D]) \le 3n^{1.9}$.
Note that each vertex in $\tilde{B}_1$ corresponds to at least $0.1n$ edges in $G[B_1,V\setminus B_{1}]$, then $|\tilde{B}_1| \le 30n^{0.9}$.
Similarly, we have $|\tilde{B}_2| \le 30n^{0.9}$.

	Then we claim that, there exists three vertices $v_1, v_2, v_3 \in B_1$~(or equivalently, $B_2$) with the same degree which is at least $0.6n$.
	Otherwise, consider the sum of degrees of vertices in $B_1$.
	\begin{equation*}
	\begin{aligned}
		\sum_{v \in B_1} d(v) & = \sum_{v \in \tilde{B}_1} d(v) + \sum_{v \in B_1 \setminus \tilde{B}_1} d(v) \\
		& \le 60n^{1.9} + \left(1.1n + 1.1n + (1.1n-1) +(1.1n-1) + \ldots + 0.6n +  0.6n\right) \\
		& \le 60n^{1.9} + \frac{(1.1n + 0.6n)n}{2}\\
		& \le 0.85n^2 + 60n^{1.9}.
	\end{aligned}
	\end{equation*}
	Similarly, we have $\sum_{v \in B_2} d(v)\le 0.85n^2 + 60n^{1.9}.$ And the vertices outside $B_1\cup B_2$ are contained in $B_{12}\cup D\cup \{u_1,u_2\}$, which has size at most $2n^{0.95}$, and the sum of degrees of these vertices is at most $4n^{1.95}$.
	Recall that $S$ is the sum of all degrees of vertices in $G$. Then we have that
	\begin{equation*}
	\begin{aligned}
		S & \le \sum_{v \in B_1} d(v) + \sum_{v \in B_2} d(v) + 4n^{1.95}\\
		&\le 2(0.85n^2 + 60n^{1.9}) + 4n^{1.95}\\
		&\le 2n^2 + n
	\end{aligned}
	\end{equation*}
	when $n$ is sufficiently large, a contradiction with (\ref{eq: sum of degrees lower bound}).

	Therefore, we may assume that there exist three vertices $v_1, v_2, v_3 \in B_1$ with the same degree which is at least $0.6n$.
	Recall that the number of vertices in $B_{12}\cup D\cup \{u_1, u_2\}$ is at most $2n^{0.95}$.
	Then for each $v_i$~($i=1,2,3$), it either has at least $0.2n$ neighbors in $B_1$ or has at least $0.2n$ neighbors in $B_2$.
	By the pigeonhole principle, without loss of generality, we may assume that $v_1$ and $v_2$ have at least $0.2n$ neighbors in $B_1$.
	Denote the $0.2n$ neighbors of $v_1$ and $v_2$ in $B_1$ by $A_1$ and $A_2$ respectively.
	If $\left| A_1 \cap A_2 \right| \ge 0.1n$, then $A_1 \cap A_2$ is $P_{2\ell}$-free.
	By Theorem~\ref{thm: Erd\H{o}s and Gallai}, we have at least $\binom{0.1n}{2} - 2\ell n$ missing edges in $G[A_1 \cap A_2]$, that is $|\mathcal{M}[B_1]| \ge 0.004n^2$.
	If $\left| A_1 \cap A_2 \right| < 0.1n$, then similarly, $G[A_1\setminus A_2, A_2\setminus A_1]$ is $P_{2\ell}$-free, which also leads to at least $0.004n^2$ missing edges in $G[B_1]$.
	Therefore, we have $|\mathcal{M}[B_1]| \ge 0.004n^2$ in all cases.
	Combining (\ref{eq: prove lb of D}) and $b_1 = b_2 \le n$, we have a contradiction when $n$ is sufficiently large.
\end{proof}

\begin{claim}\label{claim: v1 v2 in B12}
	There exist two distinct vertices $v_1, v_2 \in B_{12}$ such that $d_D(v_1) = d_D(v_2)$.
\end{claim}

\begin{proof}
By (\ref{eq: using c}), $b_{12} - d = 1+2c$.
For every vertex $v \in B_{12}$, $0 \le d_D(v) \le d$.
By pigeonhole principle, the statement fails only when $c = 0$, that is, $\beta = n + 1$ and $\alpha = 1$, and $\{d_D(v) \mid v \in B_{12}\} = \{0,1,\ldots, d\}$.
In this case, let $v_i$ be the vertex in $B_{12}$ with $d_D(v_i) = d-i$, $i=0,1,\ldots,d$.
We claim that $G[D]$ is $P_{2\ell-2}$-free; otherwise $u_1$ and $u_2$ are connected by a $P_{2\ell+2}$.
By Theorem~\ref{thm: Erd\H{o}s and Gallai}, $e(G[D]) \le 2\ell n$.
Note that $b_1 = b_2 = n-d-1-c$ by (\ref{eq: using c}).
Suppose for the contrary that for every two distinct vertices $v_1,v_2\in B_{12}$, we have  $d_D(v_1)\neq d_D(v_2)$. Then we have $e(G[B_{12},D])\leq \frac{d(d+1)}{2}$. Thus $e(G[B,D])=e(G[B_{12},D])+e(G[B_1,D])+e(G[B_2,D])\leq \frac{d(d+1)}{2}+d(b_1+b_2)$.
Then we have that
\begin{equation*}
\begin{aligned}
	n^2 + n \le e(G) & \le 2\beta + e(G[B_1]) + e(G[B_2]) + 10\ell n  + e(G[B,D]) \\
	& \le 4n+ \binom{b_1}{2}+\binom{b_2}{2} + 10\ell n + \frac{d(d+1)}{2} + d(b_1+b_2)\\
	& \le  2\binom{n-d-1}{2} + (10\ell+4) n + \frac{d(d+1)}{2} + 2d(n-d-1)\\
	& \le n^2  - \frac{d^2}{2} + (10\ell+4) n.
\end{aligned}
\end{equation*}
By Claim~\ref{claim: lower bound of D}, we have  a contradiction.
\end{proof}

By Claim \ref{claim: v1 v2 in B12}, let $\gamma$ be the maximum integer such that there exist vertices $v_1, v_2 \in B_{12}$ with $d_D(v_1) = d_D(v_2) = \gamma$.
Define
\begin{equation*}
\begin{aligned}
	Y_{12} &= N(v_1) \cap N(v_2) \cap D,\\
	Y_1 &= \left(N(v_1) \cap D\right) \setminus N(v_2),\\
	Y_2 &= \left(N(v_2) \cap D\right) \setminus N(v_1).\\
\end{aligned}
\end{equation*}
And let $Y=Y_1\cup Y_2\cup Y_{12}$.
See Figure~\ref{fig: Y}.
Following the ideas in~\cite{2025arXiv250500523L}, let us first estimate the number of missing edges in $G$.
\begin{equation}\label{eq: M[B]}
	\begin{aligned}		
		|\mathcal{M}[B \cup \{u,v\}]| & \ge 2(\beta - \alpha - b_{12})b_{12} + \binom{b_{12}}{2} + 2(\beta - \alpha - b_{12}) + (\beta-\alpha - b_{12})^2 - 8\ell b\\
		& \ge (\beta +1 - \alpha)^2 - \frac{1}{2}(b_{12}^2 + 5b_{12}) - 9\ell b\\
		& \ge (c+n)^2 - \binom{b_{12}+1}{2} - 24\ell n.
	\end{aligned}
	\end{equation}
	Note that $\gamma \le d = b_{12}-2c-1$ and $b_{12} \le \beta - \alpha = n+c$. Then we have that
	\begin{equation}\label{eq: M[B12,D]}
	\begin{aligned}
		|\mathcal{M}[B_{12}, D]|  &\ge b_{12}d - (d + (d - 1) + \ldots +(d-(d-\gamma+1))+ \gamma  + \ldots + \gamma ) \\
		& = b_{12}d - \frac{(d+\gamma+1)(d-\gamma)}{2} - \gamma(b_{12}-(d-\gamma)) \\
		& = \binom{b_{12}+1}{2} - \binom{\gamma}{2} -2c^2 - c - (2c+1)\gamma - b_{12}\\
		& \ge \binom{b_{12}+1}{2} - \binom{\gamma}{2} +2c^2+3c+1-2(c+1)b_{12}\\
		& \ge \binom{b_{12}+1}{2} - \binom{\gamma}{2} -2nc - 2n.
 	\end{aligned}
	\end{equation}
	Similarly as in $B$, we have that $G[Y_{12}]$, $G[Y_1, Y_{12}]$, $G[Y_2, Y_{12}]$ and $G[Y_1, Y_2]$ are all $P_{2\ell-2}$-free, and each of them has at most $\ell n$ edges.
	Therefore,
	\begin{equation}\label{eq: M[Y]}
	\begin{aligned}
	\left|\mathcal{M}[D] \right|\geq \left|\mathcal{M}[Y]\right| & \ge \binom{y_{12}}{2} + 2y_{12}(\gamma - y_{12}) + (\gamma-y_{12})^2 - 8\ell n \\
		& \ge \gamma^2 - \frac{1}{2}y_{12}^2 - \frac{1}{2}y_{12} - 8\ell n\\
		& \ge \binom{\gamma}{2} - 8\ell n,
	\end{aligned}
	\end{equation}
	where the last inequality is from $y_{12} \le \gamma$.
	Combining (\ref{eq: missing edges}), (\ref{eq: M[B]}), (\ref{eq: M[B12,D]}) and (\ref{eq: M[Y]}), we have that
	\begin{equation}\label{eq: original extra linear term}
	\begin{aligned}
		n^2 \ge |\mathcal{M}| & \ge c^2 + n^2 - 29\ell n.
	\end{aligned}
	\end{equation}

\begin{figure}
	\centering
	\includegraphics[width=0.5\textwidth]{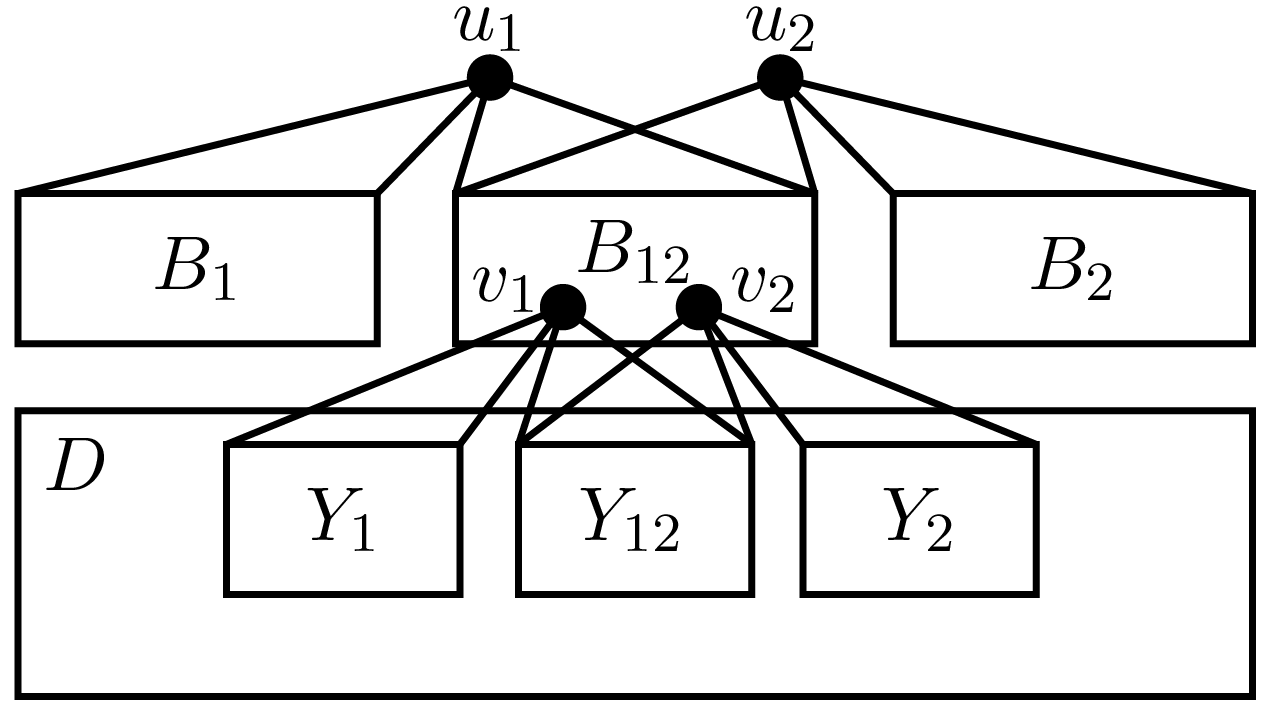}
	\caption{An illustration of the sets $Y_{12}$, $Y_1$ and $Y_2$.}\label{fig: Y}
\end{figure}

By a finer analysis, we have the following claim.
\begin{claim}\label{claim: lower bound of gamma y}
	For some fixed large constants $C = C(\ell)$ depending on $\ell$ only, we have that $c \le C n^{0.5}$,  $y_{12} \ge \gamma - C n^{0.5}$, and $\gamma \ge d - 2C n^{0.5}$.
\end{claim}

\begin{proof}
	By (\ref{eq: original extra linear term}), we easily have that $c \le C n^{0.5}$ where $C$ is large enough.

	If $y_{12} < \gamma - C n^{0.5}$, then clearly $\gamma \ge Cn^{0.5}$, and we can get a refined bound for $|\mathcal{M}[Y]|$ from (\ref{eq: M[Y]}). That is
	\begin{equation}\label{eq: M[Y] 2}
	\begin{aligned}
		|\mathcal{M}[Y]| \ge & \gamma^2 - \frac{(\gamma - C n^{0.5})^2}{2} - \frac{\gamma - C n^{0.5}}{2} - 8\ell n \\
		& \ge \binom{\gamma}{2} + \frac{C^2}{2} n - 8\ell n.
	\end{aligned}
	\end{equation}
		Then combining (\ref{eq: M[B]}), (\ref{eq: M[B12,D]}) and (\ref{eq: M[Y] 2}), we have that
	\begin{equation}\label{eq: lb M}
	\begin{aligned}
		 |\mathcal{M}| & \ge c^2 + n^2 - 33\ell n  +\frac{C^2}{2}n.
	\end{aligned}
	\end{equation}
	Since $C$ is a large constant, we have a contradiction with (\ref{eq: missing edges}).

	Now we focus on the third statement.
	Suppose otherwise, $\gamma < d - 2C n^{0.5}$.
	Let $Z = D\setminus Y$, and define
	\begin{equation*}
	\begin{aligned}
		Z_1 &= \left\{w \in Z \mid d_Y(w) \le \frac{1}{2}y +2\ell \right\},\\
		Z_2 &= Z\setminus Z_1.
	\end{aligned}
	\end{equation*}
	By assumption, we have that $$z = d - y=d-(2\gamma-y_{12})=d-\gamma-(\gamma-y_{12}) \ge 2C n^{0.5}-C n^{0.5}=C n^{0.5}.$$
	We first claim that $z_2 \le \frac{1}{2}C n^{0.5}$.
	Suppose otherwise, $z_2 > \frac{1}{2}C n^{0.5}$.
	If $\mathcal{M}[Z_2] \ge \frac{1}{8}C n$, then we can add $\frac{1}{8}C n$ more missing edges to $\mathcal{M}[D]$ in (\ref{eq: M[Y]}). Combining (\ref{eq: M[B]}), (\ref{eq: M[B12,D]}), (\ref{eq: M[Y]}) and (\ref{eq: original extra linear term}), we have $|\mathcal{M}|> n^2$, a contradiction with (\ref{eq: missing edges}).
	Therefore, there exists an edge in $G[Z_2]$, say $w_1w_2$.
	Choose arbitrary vertices $w_3, w_4,\ldots,w_{\ell-1} \in Z_2$.
	By the definition of $Z_2$, for every $i=2,3,\ldots,\ell-2$, $w_i$ and $w_{i+1}$ have a distinct common neighbor in $Y$, say $t_i$~(see Figure~\ref{fig: Z2}).

	% Note that $y \ge \gamma \ge d - C n^{0.5} \ge n^{0.8}$ by Claim~\ref{claim: lower bound of D}.
	% Then we have that
	% \begin{equation*}
	% \begin{aligned}
	% 	y-y_{12} & = 2\gamma - 2y_{12} \le 2Cn^{0.5} \le 0.1y.
	% \end{aligned}
	% \end{equation*}
	Since $d_Y(w_1) \ge \frac{1}{2}y + 2\ell$, $w_1$ has a neighbor $t_1$ in $Y_1 \cup Y_{12}$ and similarly, $w_{\ell-1}$ has a neighbor $t_{\ell-1}$ in $Y_2 \cup Y_{12}$ such that $t_1,t_2,\ldots,t_{\ell-1}$ are distinct vertices.
	Then we have a path $u_1v_1t_1w_1w_2t_2w_3\ldots w_{\ell-1}t_{\ell-1}v_2u_2$ of length $2\ell -1$ connecting $u_1$ and $u_2$~(see Figure~\ref{fig: Z2}), a contradiction.

	\begin{figure}
		\centering
		\includegraphics[width=0.4\textwidth]{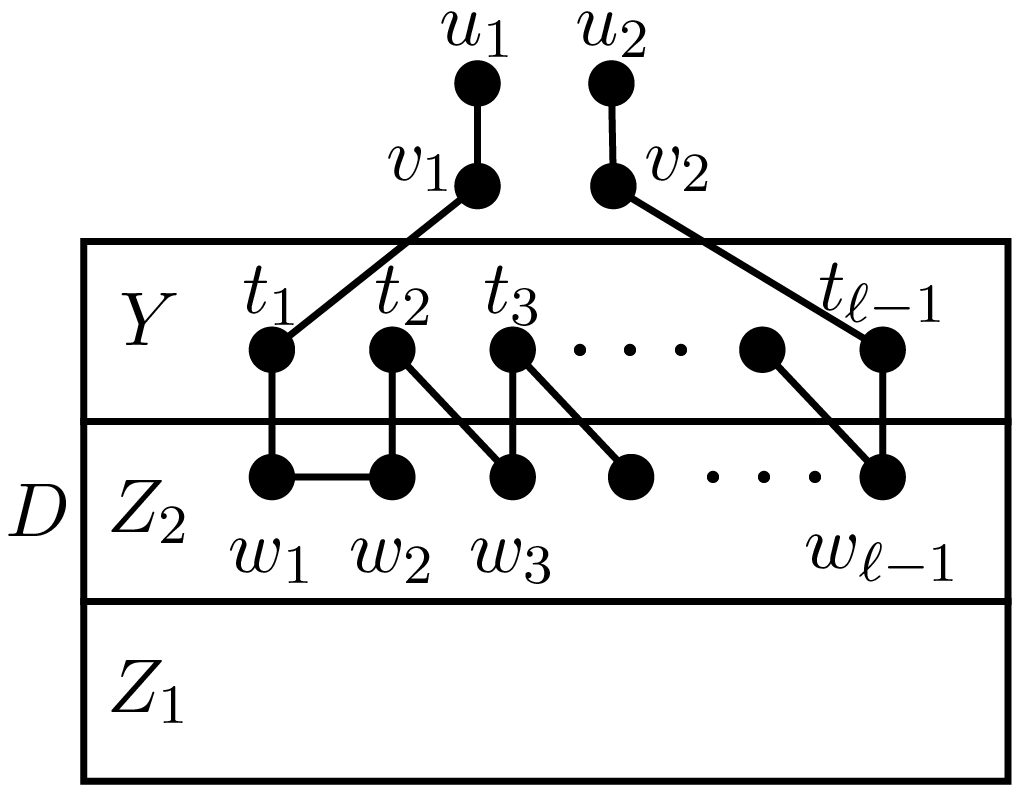}
		\caption{An illustration of the path $u_1v_1t_1w_1w_2t_2w_3\ldots w_{\ell-1}t_{\ell-1}v_2u_2$.}\label{fig: Z2}
	\end{figure}

	Now we have that $z_2 \le\frac{1}{2} C n^{0.5}$. Then $z_1 = z - z_2 \ge \frac{1}{2} C n^{0.5}$ by the lower bound of $z$.
	By the definition of $Z_1$, we have that $|\mathcal{M}[Z_1, Y]| \ge \left\lfloor \frac{1}{2} y - 2\ell \right\rfloor z_1$.
	If $\gamma > n^{0.5}$, then $|\mathcal{M}[Z_1, Y]|  \ge \gamma z_1/4 \ge C n/8$, and we can also add this term to $\mathcal{M}[D]$, and get a contradiction  with (\ref{eq: missing edges}) by (\ref{eq: M[B]}), (\ref{eq: M[B12,D]}), (\ref{eq: M[Y]}) and (\ref{eq: original extra linear term}) when $C$ is large enough.
	Then we conclude that $\gamma \le n^{0.5}$. We complete the proof by considering two cases.

	\textbf{Case 1:} $\beta \ge 2b_{12}$.
	In this case, we partition $Z$ in another way.
	Define
	\begin{equation*}
	\begin{aligned}
		Z_1' &= \left\{w \in Z \mid d_{B}(w) \ge \frac{1}{2}b + 2\ell \right\},\\
		Z_2' &= Z\setminus Z_1'.
	\end{aligned}
	\end{equation*}

	We claim that $z_1' \le Cn^{0.5}$.
	Suppose $z_1' > Cn^{0.5}$. Then we claim that $G[Z_1']$ cannot contain any edge.
	The statement can be proved by a similar argument as before.
	Then $|\mathcal{M}[Z_1']| \ge \binom{z_1'}{2} \ge \frac{C^2}{4}n$, and we can also add this term to $\mathcal{M}[D]$, and get a contradiction   with (\ref{eq: missing edges}) by (\ref{eq: M[B]}), (\ref{eq: M[B12,D]}), (\ref{eq: M[Y]}) and (\ref{eq: original extra linear term}) when $C$ is sufficiently large.
	
	Then  $z_2' = d - \gamma - z_1' \ge n^{0.8}$ by Claim \ref{claim: lower bound of D} and $\gamma \le n^{0.5}$.
	For each $w \in Z_2'$, it has at least $\frac{1}{2}b - 2\ell$ missing edges in $B$.
	Since $b_{12} \le \beta/2$, we have that $\frac{1}{2} b - b_{12} \ge \frac{1}{4}\beta \ge \frac{1}{4}n$.
	Then $|\mathcal{M}[B\setminus B_{12}, Z_2']| \ge \frac{1}{4}n z_2' \ge \frac{1}{4}n^{1.8}$.
	We can also add this term to $\mathcal{M}[D]$, and get a contradiction   with (\ref{eq: missing edges}) by (\ref{eq: M[B]}), (\ref{eq: M[B12,D]}), (\ref{eq: M[Y]}) and (\ref{eq: original extra linear term}).

	\textbf{Case 2:} $n < \beta < 2b_{12}$.
	Let us partition $Z$ in another way.
	Define
	\begin{equation*}
	\begin{aligned}
		Z_1'' &= \{w \in Z \mid d_{B_{12}}(w) \ge 2\ell \},\\
		Z_2'' &= Z\setminus Z_1''.
	\end{aligned}
	\end{equation*}
	It is clear that $G[Z_1'']$ is $P_{2\ell-2}$-free. If $z_1'' > Cn^{0.5}$, then $|\mathcal{M}[Z_1'']|\ge 1/2C^2n-O(n^{0.5})$. We can  add this term to $\mathcal{M}[D]$, and get a contradiction   with (\ref{eq: missing edges}) by (\ref{eq: M[B]}), (\ref{eq: M[B12,D]}), (\ref{eq: M[Y]}) and (\ref{eq: original extra linear term}).
So we have that $z_1'' \le Cn^{0.5}$.
	Recall that $d = b_{12} - 2c - 1$ from (\ref{eq: using c}) and $c\le Cn^{0.5}$ by Claim \ref{claim: lower bound of gamma y}. Then we have a better estimation of $\mathcal{M}[B_{12}, D]$:
	\begin{equation}
	\begin{aligned}
		|\mathcal{M}[B_{12}, D]| & \ge z_2''(b_{12} - 2\ell)\\
		& \ge \left(d - 2\gamma - z_1'' \right) (b_{12} - 2\ell) \\
		& \ge \left(b_{12} - ( 2 + C) n^{0.5} -2c-1 \right) (b_{12} - 2\ell) \\
		& \ge \binom{b_{12}+1}{2} + b_{12}^2/2  - O(n^{1.5}) \\
		& \ge \binom{b_{12}+1}{2} + n^2/8 - O(n^{1.5}).
 	\end{aligned}
	\end{equation}
	We have an extra $n^2$-term compared to (\ref{eq: M[B12,D]}), a contradiction   with (\ref{eq: missing edges}) by combining (\ref{eq: M[B]}) and (\ref{eq: M[Y]}).
\end{proof}

Claim~\ref{claim: lower bound of gamma y} shows that within a $O(n^{0.5})$-gap, $\beta \approx n$, $y_{12} \approx \gamma$ and $\gamma \approx d$.
% This gives a good characterization of the structure of $D$.
These properties are used to characterize the structure of $B$.
Define
\begin{equation*}
\begin{aligned}
	B' &= \left\{w \in B \mid d_D(w) \ge \frac{1}{2}d + 2\ell \right\}.	
\end{aligned}
\end{equation*}
Let $b'=|B'|$. We would like to show that $b \approx b'$.

\begin{claim}\label{claim: B setminus B'}
	$|B \setminus B' | \le 6C n^{0.6}$, where $C = C(\ell)$ is given in Claim \ref{claim: lower bound of gamma y}.
\end{claim}

\begin{proof}
	For each $w \in B\setminus B'$, it causes at least $\frac{1}{2}d - 2\ell \ge \frac{1}{3}n^{0.9}$ missing edges in $G[B,D]$.

	Note that in (\ref{eq: M[B12,D]}),
	\begin{equation*}
	\begin{aligned}
				 & b_{12}d - (d + (d-1) + \ldots + \gamma + \ldots + \gamma )\\
		\le & (d-\gamma)b_{12}\\
		\le & Cn^{1.5},
	\end{aligned}
	\end{equation*}
	where the last inequality is from Claim~\ref{claim: lower bound of gamma y}.

	Suppose otherwise, $|B \setminus B' | >6C n^{0.6}$. By Claim \ref{claim: lower bound of D} and $\frac{1}{2}d - 2\ell \ge \frac{1}{3}n^{0.9}$,
	 we have that
	\begin{equation}\label{eq: M[B,D] 2}
	\begin{aligned}
		|\mathcal{M}[B,D]| & \ge 6C n^{0.6} \times\frac{1}{3}n^{0.9} \\
& \ge  C n^{1.5} +  b_{12}d - (d + (d-1) + \ldots + \gamma + \ldots + \gamma )\\
		& \ge  C n^{1.5} + \binom{b_{12}+1}{2} - \binom{\gamma}{2} - 2nc - 2n.\\
		% \le & b_{12}d - (d + (d-1) + \ldots + \gamma + \ldots + \gamma ).
	\end{aligned}
	\end{equation}
We replace (\ref{eq: M[B12,D]}) with (\ref{eq: M[B,D] 2}), and combine (\ref{eq: M[B]}), (\ref{eq: M[B,D] 2}) and (\ref{eq: M[Y]}), we have that
	\begin{equation*}
	\begin{aligned}
		|\mathcal{M}| &\ge n^2 + c^2 - 32\ell n + C n^{1.5}.
	\end{aligned}
	\end{equation*}
	Since $C$ is large enough, we have $|\mathcal{M}| > n^2$, a contradiction with (\ref{eq: missing edges}).
\end{proof}

Define
\[
D' = \left\{w \in D \mid d_B(w) \ge \frac{1}{2} b +2\ell \right\}.
\]

\begin{claim}\label{claim: D setminus D'}
	$|D\setminus D'| \le C n^{0.6}$, where $C = C(\ell)$ is given in Claim \ref{claim: lower bound of gamma y}.
\end{claim}

\begin{proof}
	Similarly, each $w \in D\setminus D'$ causes $\frac{1}{2} b - 2\ell \ge \frac{1}{3}n$ missing edges in $G[B, D]$.
	By a similar argument as in the proof of Claim~\ref{claim: B setminus B'}, we can prove the statement.
	The details are omitted here.
\end{proof}

We further partition $D\setminus D'$ into $D_1$ and $D_2$ as follows:
\begin{equation*}
\begin{aligned}
	D_1 &= \{w \in D \setminus D' \mid \left|N(w) \cap D' \right| \ge 2 \},\\
	D_2 &= D \setminus (D' \cup D_1).
\end{aligned}
\end{equation*}

Claims~\ref{claim: B setminus B'} and \ref{claim: D setminus D'} follow from counting the missing edges in $G$.
Using the property that no vertices with same degree are connected by a $P_{2\ell+2}$, we have the following claim.

\begin{claim}\label{claim: empty sets}
	For every vertex $t \in B'$, $N(t)\cap B=\emptyset$. Moreover
	$G[D']$, $G[D_1]$, and $G[D_1,B]$ are all empty.
\end{claim}

\begin{proof}
	Firstly, let us prove the first statement.
	Suppose $t_2\in B'$ with a neighbor $t_1 \in B$.
	Without loss of generality, we may assume that $t_1 \in B_1 \cup B_{12}$. Then $u_1t_1\in E(G)$~(see Figure~\ref{fig: B' has neighbor in B}).
	By Claim~\ref{claim: B setminus B'} and $\beta>n$, there exist distinct vertices $t_3,\ldots,t_{\ell+1} \in (B_{12}' \cup B_2')\setminus\{t_1,t_2\}$.
	By the definition of $B'$, for every $i=2,3,\ldots,\ell$, $t_i$ and $t_{i+1}$ have a common neighbor $w_i\in D$ such that the vertices $w_2,w_3,\ldots,w_{\ell}$ are all different.
	Then $u_1$ and $u_2$ are connected by a $P_{2\ell+2}=u_1t_1t_2w_2\ldots t_{\ell}t_{\ell+1}u_2$, a contradiction.

	\begin{figure}[h]
		\begin{minipage}{0.5\textwidth}
			\centering
			\includegraphics[width=\textwidth]{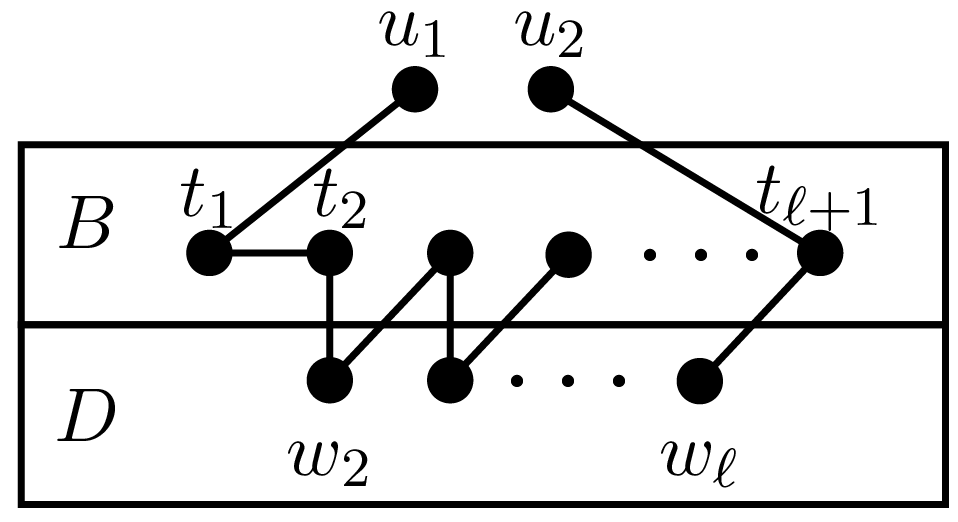}
			\caption{The case when $t_2\in B'$ has a neighbor $t_1\in B$.}\label{fig: B' has neighbor in B}
		\end{minipage}
		\begin{minipage}{0.5\textwidth}
			\centering
			\includegraphics[width=\textwidth]{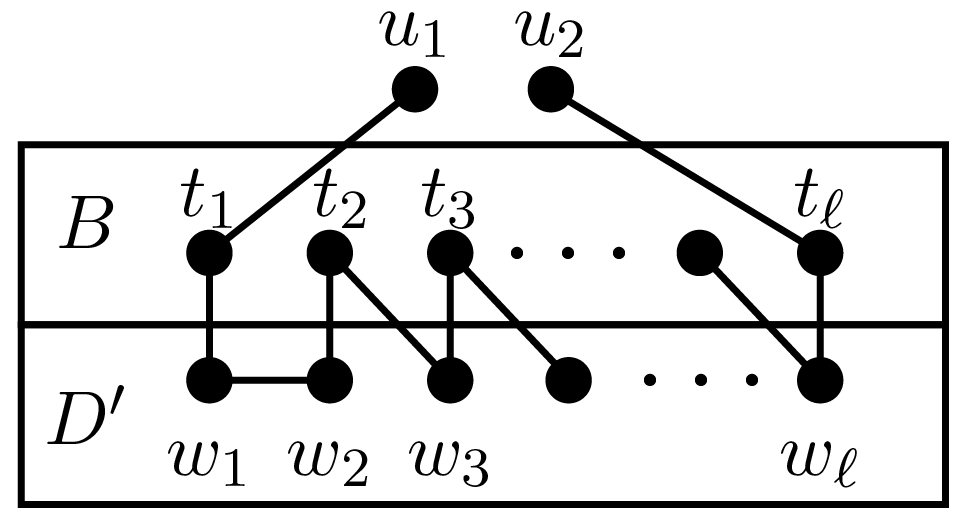}
			\caption{The case when $G[D']$ is not empty with an edge $w_1w_2$.}\label{fig: G[D'] empty}
		\end{minipage}
	\end{figure}
	\begin{figure}
		\begin{minipage}{0.5\textwidth}
			\centering
			\includegraphics[width=\textwidth]{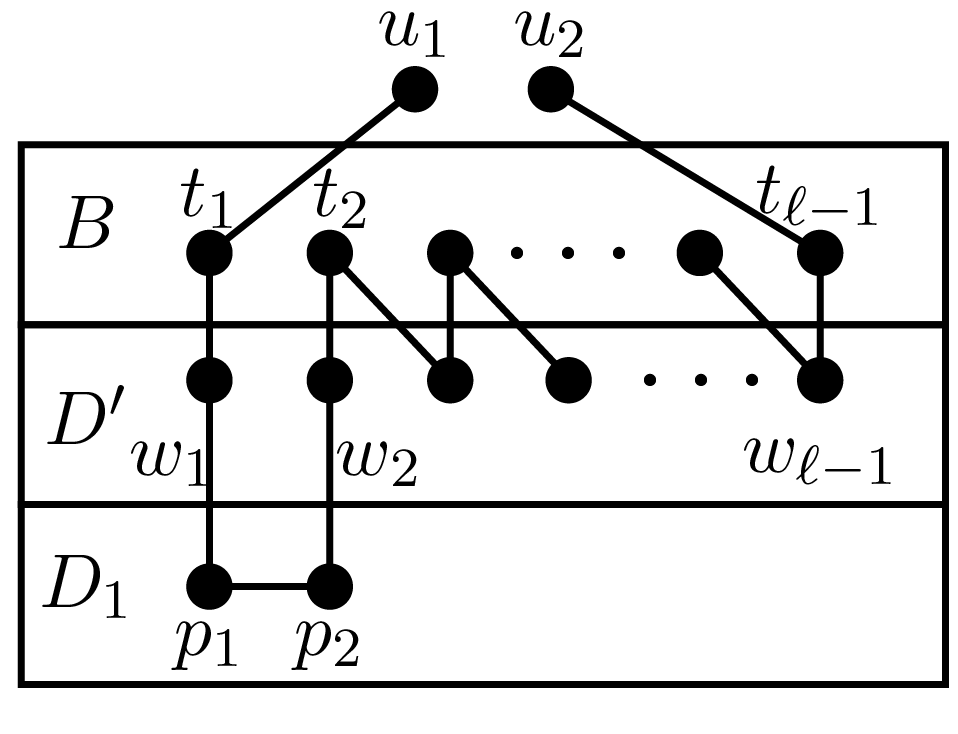}
			\caption{The case when $G[D_1]$ is not empty.}\label{fig: G[D_1] empty}
		\end{minipage}
		\begin{minipage}{0.5\textwidth}
			\centering
			\includegraphics[width=\textwidth]{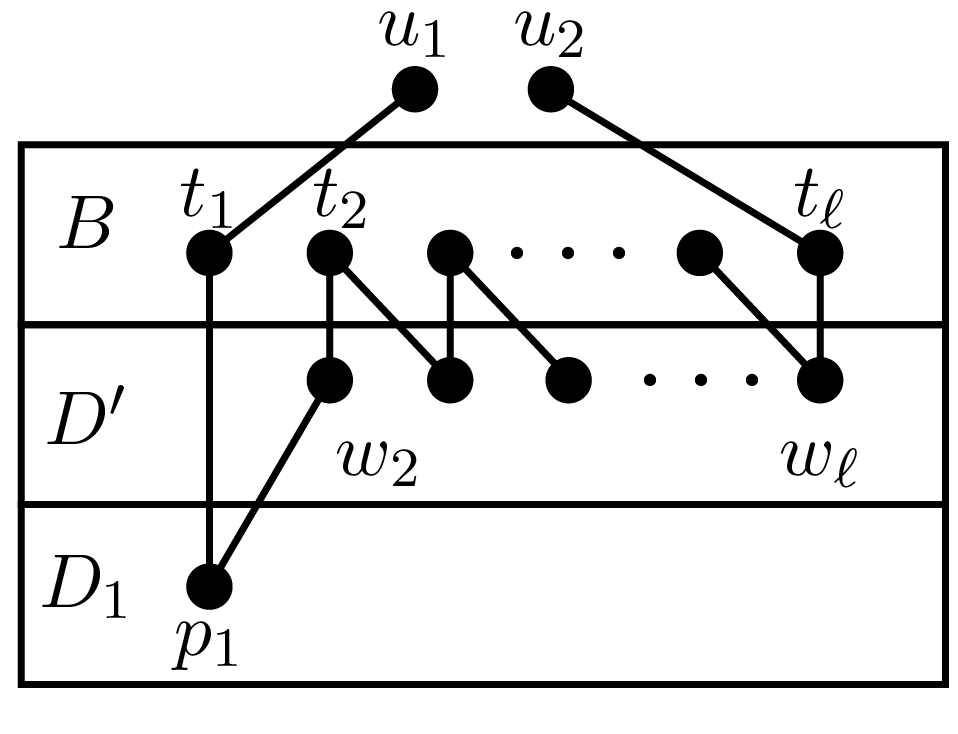}
			\caption{The case when $G[D_1,B]$ is not empty.}\label{fig: G[D_1,B] empty}
		\end{minipage}
	\end{figure}

	Secondly, we claim that $G[D']$ is empty.
	Suppose  there exist two adjacent vertices $w_1,w_2 \in D'$~(see Figure~\ref{fig: G[D'] empty}).
	By the definition of $D'$, we can assume $w_1$ has a neighbor $t_1$ in $B_{12} \cup B_1$.
	Moreover, by Claims~\ref{claim: lower bound of D} and \ref{claim: D setminus D'}, there exist vertices $w_3,\ldots,w_{\ell} \in D'\setminus\{w_1,w_2\}$. By the definition of $D'$, $w_i,w_{i+1}$ have a common neighbor $v_i$ in $B$ for each $i=2,3,\ldots,\ell-1$ and the vertices $t_1,t_2,\ldots,t_{\ell-1}$ are all different.
	Finally, $w_{\ell}$ has a neighbor $t_{\ell}$ in $B_{12} \cup B_2$ avoiding the vertices $t_1,t_2,\ldots, t_{\ell-1}$.
	Then $u_1$ and $u_2$ are connected by a $P_{2\ell+2}=u_1t_1w_1w_2t_2\ldots w_{\ell}t_{\ell}u_2$, a contradiction.
	Then we conclude that $G[D'] = \emptyset$.

	Then, we claim that $G[D_1]$ is empty.
	Suppose  there exist two adjacent vertices $p_1,p_2 \in D_1$~(see Figure~\ref{fig: G[D_1] empty}).
	By the definition of $D_1$, $p_1$ has a neighbor $w_1 \in D'$, and $p_2$ has a neighbor $w_2 \in D'$ such that $w_1 \neq w_2$.
	Similarly as the case before, there exist vertices $w_3,\ldots,w_{\ell-1} \in D'$ and $t_1,t_2,\ldots,t_{\ell-1} \in B$ such that $u_1,u_2$ are connected by a $P_{2\ell+2}=u_1t_1w_1p_1p_2w_2t_2\ldots w_{\ell-1}t_{\ell-1}u_2$, a contradiction.

	By the same argument, we can prove that $G[D_1,B]$ is empty~(see Figure~\ref{fig: G[D_1,B] empty}).
\end{proof}

Now we are ready to conclude the $\beta > n$ case.
% Let $E_{D_2, B'}$ be the set of edges in $G$ with one endpoint in $D_2$ and the other endpoint outside of $D_2$.
By the above claims, we can calculate the sum of degrees of vertices in $G$ as follows:
\begin{enumerate}
	\item The sum of degrees of vertices in $B'$ is at most $b'(d' + d_2+2)$ by Claim~\ref{claim: empty sets}.
	\item The sum of degrees of vertices in $B\setminus B'$ is at most $(b-b')(\frac{1}{2}d + 2\ell + 2 + b-b')$ by the definition of $B'$.
	\item The sum of degrees of vertices in $D'$ is at most $d'(b + d_1 + e(G[D_2, D'])) \le d'(b+d_1) + d_2$ by Claim~\ref{claim: empty sets}.
	\item The sum of degrees of vertices in $D_1$ is at most $d_1(d' + d_2)$ by Claim~\ref{claim: empty sets}.
	\item The sum of degrees of vertices in $D_2$ is at most $d_2(1 +d_2 + d_1 + \frac{1}{2}b +2\ell)$ by the definition of $D'$ and $D_2$.
\end{enumerate}

Put the above together, then we have that
\begin{equation*}
\begin{aligned}
	S &\le 2\beta + b'(d' + d_2 +2) + (b-b')\left(\frac{1}{2}d + 2\ell+ 2 + b-b'\right) \\
	&\quad + d'(b+d_1) + d_1(d'+d_2) + d_2\left(1 + d_2 + d_1 + \frac{1}{2}b + 2 + 2\ell\right)\\
	& \le 2(b-\alpha) + b(d - d_1 +2) + (b-b') \left(-\frac{1}{2}d + 2\ell + d_1 + b-b'\right)  \\
	& \quad + (d'+d_2)(b+2d_1) + d_2\left(-\frac{1}{2}b + d_2 +2+2\ell\right) \\
	& \le 2(b-\alpha) + b(d - d_1 +2) + (d'+d_2)(b+2d_1),
\end{aligned}
\end{equation*}
where the last inequality holds since both $-\frac{1}{2}d + d_2 + 2\ell + b-b'$ and $-\frac{1}{2}b + d_2 + 2 + 2\ell$ are negative by Claims~\ref{claim: lower bound of D} and \ref{claim: D setminus D'}.
Let $x_1 = d' +d_2 = d-d_1, x_2 = d_1 + b$. Then we have that
\begin{equation*}
\begin{aligned}
	S & \le 2x_1 x_2 + 4b -2\alpha
	 \le 2x_2(x_1+2) - 2\alpha.
\end{aligned}
\end{equation*}
Since $x_2 + x_1 + 2 = 2n+1$, we have that $S \le 2n(n+1)$.
From the inequalities, $S = 2n(n+1)$ holds only when $\alpha =0$, $b=b'$ and $\{b, d'+2\} = \{n,n+1\}$.
Then by Claim~\ref{claim: empty sets}, it is easy to show that $G=K_{n,n+1}$.

% \subsection{try $\ell=3$ new method}

% When $\ell = 3$, we have that

% \begin{equation*}
% \begin{aligned}
% 	e(G) & \le 2 \beta - \alpha + e(G[B_1]) + e(G[B_2]) + 4n + e(G[D_2]) + e(G[D_1, D']) + e(G[D_2, B \cup D'
% 	\cup D_1])
% 	\\
% 	& \le 6n + \binom{b_1}{2} + \binom{b_2}{2} + d' b + \binom{d_2}{2} + d_1d' + d_2\left(\frac{1}{2}b + d_1 + 2\ell \right)\\
% 	& \le 10\ell n^{1.5} + n^2 + n(-d+d'-d_1) + d^2 + \frac{d'^2}{2} - \frac{3dd'}{2} - \frac{d_1^2}{2} + d_1(d/2+d')
% \end{aligned}
% \end{equation*}

% ------------------------------------------------------------
\section{Proof of Theorem \ref{thm: main 2n+1} in the \texorpdfstring{$\beta \le n$}{β ≤ n} case}\label{sec: beta < n}

\begin{claim}\label{claim: Delta > n + sqrt n} We have
	$\Delta \ge n + \sqrt{n}$.
\end{claim}
\begin{proof}
	Suppose  $\Delta < n + \sqrt{n}$. By the definition of $\beta$, we have that
	\begin{equation*}
	\begin{aligned}
		S & \le \Delta + (\Delta -1 ) + \ldots + \beta + \beta+ \ldots + \beta\\
			& \le \frac{(\Delta+\beta)(\Delta - \beta + 1)}{2} + \beta(2n+1 - (\Delta -\beta + 1))\\
			& \le \frac{(\Delta+n)(\Delta - n + 1)}{2} + n(3n - \Delta  )\\
			& \le \frac{(2n+\sqrt{n})(\sqrt{n}+1)}{2} + n(2n -\sqrt{n})\\
			& \le 2n^2 + \frac{3}{2}n + \frac{1}{2}\sqrt{n} < 2n^2 + 2n,
	\end{aligned}
	\end{equation*}
	 when $n$ is large enough, a contradiction with (\ref{eq: sum of degrees lower bound}).
\end{proof}

	\begin{figure}
		\centering
		\includegraphics[width=0.7\textwidth]{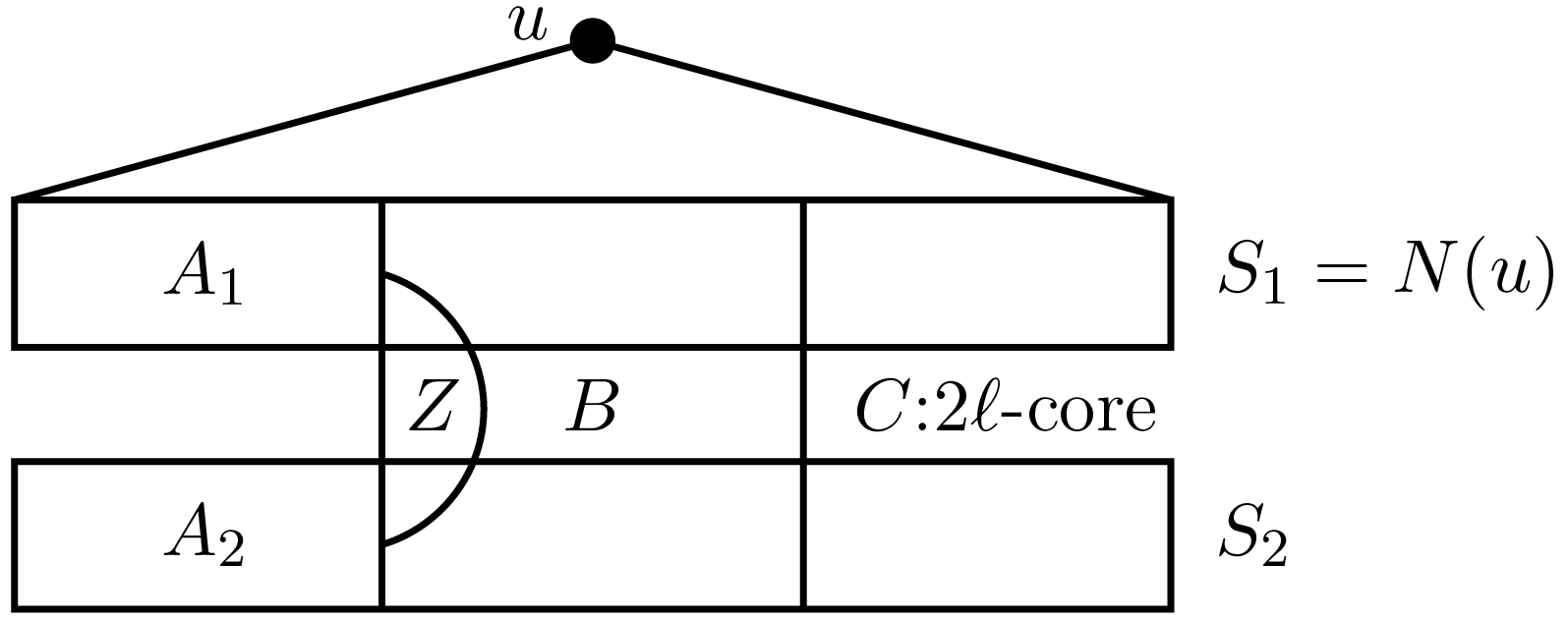}
		\caption{An illustration of the vertex $u$ and the sets $S_1,S_2,A,B,C,Z$.}\label{fig: Delta > n}
	\end{figure}

Let $u\in V(G)$ with $d(u)=\Delta$.
Let $S_1 = N(u), S_2 = V \setminus N[u]$.
It is clear that $|S_1| = \Delta, |S_2| = 2n-\Delta$. Set

\begin{equation*}
\begin{aligned}
	& A_1 = \{ v \in S_1 \mid d(v) \le 2n-\Delta+2\ell \}, \\
	& A_2 = \{ v \in S_2 \mid d(v) \le 2n-\Delta+2\ell-1 \}. \\
\end{aligned}
\end{equation*}
Let $A = A_1 \cup A_2$, $G_1 = G[V\setminus \left(  A \cup \{u\} \right)]$.
The \textit{$t$-core} of  $G$ is the largest induced subgraph of $G$ that has minimum degree at least $t$.
It can be obtained by repeatedly removing vertices with degree less than $t$ from $G$.

Let $C$ be the $2\ell$-core of $G_1$, and $B$ = $V(G_1) \setminus  V(C)$.
We use lower case letters to denote the size of each set, for example, $a_1 = |A_1|$, $a_2 = |A_2|$, $b = |B|$, $c = |C|$, etc. Then $a+b+c=2n$.

\begin{claim}\label{claim: different degrees in C}
	vertices in $V(C)\cup \{u\}$ have different degrees in $G$.
\end{claim}

\begin{proof}
	Suppose there exist two vertices $v_1,v_2 \in V(C)\cup \{u\}$ such that $d(v_1)=d(v_2)$.
	By the definition of a $2\ell$-core and Claim \ref{claim: Delta > n + sqrt n}, there is a path $v_1u_2u_3\ldots u_{2\ell-3}$ in $C$ avoiding $v_2$.
	By the definition of $A$, we have that $d(u_{2\ell-3}) > 2n-\Delta+2\ell$ when $u_{2\ell -3} \in A_1$, and $d(u_{2\ell-3}) > 2n-\Delta+2\ell-1$ when $u_{2\ell -3} \in A_2$.
	In either case, $u_{2\ell-3}$ has at least $2\ell$ neighbors in $S_1$ by $|S_2| = 2n-\Delta$.
	Let $w_1$ be a neighbor of $u_{2\ell-3}$ in $S_1$ avoiding the path and $v_2$.
	Similarly, $v_2$ has a neighbor $w_2$ in $S_1$ avoiding the path, $v_1$ and $w_1$.
	Then $v_1$ and $v_2$ are connected by a $P_{2\ell+2}=v_1u_2u_3\ldots u_{2\ell-3}w_1u w_2v_2$, a contradiction.
\end{proof}

Let $E_B$ be the set of edges in $G_1$ with at least one endpoint in $B$ if $B\not=\emptyset$; otherwise let $E_B=\emptyset$.
Then by the definition of $2\ell$-core, we have $|E_B| \le 2\ell b$. We assume $B\not=\emptyset$ and
let
\begin{equation*}
	Z = \{ v \in B \mid d_{G_1}(v) \ge b^{0.9}\}.
\end{equation*}

Note that each vertex in $Z$ corresponds to at least $b^{0.9}$ edges in $E_B$.
And each edge in $E_B$ corresponds at most two vertices in $Z$.
Therefore, we have $|Z| \le 4\ell b^{0.1}$.
Let $Y = B \setminus Z$.

\begin{claim}\label{claim: upper bound of a}
	$a \le n+n^{0.6}$.
\end{claim}

\begin{proof}
	Suppose  $a > n+n^{0.6}$.
	Then we have

	\begin{equation*}
	\begin{aligned}
		S &\le \Delta + \Delta (b+c) + (2n-\Delta + 2\ell)a \\
		& \le \Delta(2n+1) + a(2n-2\Delta + 2\ell)\\
		& \le \Delta(2n+1) + (n+n^{0.6})(2n-2\Delta + 2\ell)\\
		& \le (n+\sqrt{n})(2n+1) + (n+n^{0.6})(2n-2(n+\sqrt{n}) + 2\ell)\\
		& \le 2 n^{2} + n - 2 n^{1.1} + 2 \ell n^{0.6} + 2 \ell n + n^{0.5}\le 2n^2 + n
	\end{aligned}
	\end{equation*}
	 when $n$ is large enough, a contradiction  with (\ref{eq: sum of degrees lower bound}).
\end{proof}

% \begin{equation*}
% \begin{aligned}
% 	Y' = \{w \in Y \mid d(w) \le \frac{1}{2}a + 4\ell b^{0.9}\},\\
% 	Y'' = \{w \in Y \mid d(w) > \frac{1}{2}a + 4\ell b^{0.9}\},\\
% \end{aligned}
% \end{equation*}

% \begin{claim}\label{claim: c le 2a}
% 	$c \le 2a$.
% \end{claim}

% \begin{proof}
% 	Otherwise, we have $c > 2a$.
% 	Then
% 	\begin{equation*}
% 	\begin{aligned}
% 		S & \le \frac{(c+1)(\Delta -c)}{2} + 4\ell b + 2a(2n-\Delta + 2\ell)\\
% 	\end{aligned}
% 	\end{equation*}
% 	with respect to $a+b+c = 2n$, $n+\sqrt{n} \le \Delta \le 2n$.
% 	$\ell$ is a constant and $n$ is large compared to $\ell$.
% 	Then we claim $S \le 2n^2 + n$~(verified by Python, done).
% \end{proof}

% \begin{claim}\label{claim: lower bound a}
% 	$a \ge n/10$.
% \end{claim}

% \begin{proof}
% 	Suppose otherwise, $a < n/10$.
% 	By Claim~\ref{claim: c le 2a}, we have $c \le 2a < 2n/10$.
% 	Sum of degree is at most
% 	\begin{equation*}
% 	\begin{aligned}
% 		S & \le \frac{(c+1)(\Delta -c)}{2} + 4\ell b + 2a(2n-\Delta + 2\ell)\\
% 	\end{aligned}
% 	\end{equation*}
% 	with respect to $a+b+c = 2n$.
% 	Then we claim the sum of degree is at most $2n^2 + n$~(verified by Python, done).
% \end{proof}

We will finish the proof by considering two cases when $a < n$ and $ a \ge n$.

\textbf{Case 1:} $a  < n$.
In this case, let us consider the sum of degrees of vertices in $G$.
\begin{enumerate}
	\item For vertices in $V(C) \cup \{u\}$, they receive different degrees by Claim~\ref{claim: different degrees in C}. Then the sum of degrees is at most
	\begin{equation*}
	\begin{aligned}
		&\Delta + (\Delta -1) + \ldots + (\Delta - c)
	\le &\frac{(2\Delta -c)(c+1)}{2}.
	\end{aligned}
	\end{equation*}
	\item For vertices in $A$, the sum of degrees is at most $a(2n-\Delta + 2\ell)$ by the definition of $A$.
	\item For vertices in $B$,
	 $\sum_{v \in B} d(v) = \sum_{v \in B} (d_{G_1}(v) + (d(v) - d_{G_1}(v)))$. Note that for a fixed vertex $v \in B$, $d(v) - d_{G_1}(v)$ is at most $a+1$.
	 Moreover, $\sum_{v \in B} d_{G_1}(v) \le 2|E_B| \le 4\ell b$ by the definition of $2\ell$-core.
	 Then the sum of degree is at most $4\ell b + a b$. So
\end{enumerate}
\begin{equation*}
\begin{aligned}
	S &\le \frac{1}{2}(c+1)(2\Delta - c) + 4\ell b + a(2n-\Delta + 2\ell) + a b.\\
\end{aligned}
\end{equation*}

First, we view $a$ and $b$ as variables with respect to $a+b=2n-c$.
Then
\begin{equation}\label{eq: count S}
\begin{aligned}
	S & \le \frac{1}{2}(c+1)(2\Delta - c) + a(2n-c-a) + a(2n-\Delta + 2\ell) + 4\ell(2n-c-a)\\
	& \le \frac{1}{2}(c+1)(2\Delta - c) + 4\ell(2n-c) -a^2 + a(4n-\Delta -c- 2\ell).
	\end{aligned}
	\end{equation}
	Since $a<n$, we will end the proof by considering the following two subcases.

	\textbf{Subcase 1:} $(4n-\Delta -c- 2\ell)/2 \le n$.
	In this subcase, $S$ reaches the maximum when $a=(4n-\Delta-c-2\ell)/2$. Then we have
	\begin{equation*}
	\begin{aligned}
		S &\le \frac{1}{2}(c+1)(2\Delta - c) + 4\ell(2n-c) + {(4n-\Delta -c- 2\ell)}^2/4\\
		&= \frac{\Delta^{2}}{4} + \Delta \ell - 2 \Delta n + \Delta + \ell^{2} + 4 \ell n - \frac{c^{2}}{4} + c \left(\frac{3 \Delta}{2} - 3 \ell - 2 n - \frac{1}{2}\right) + 4 n^{2}.
	\end{aligned}
	\end{equation*}
If $0\le 3\Delta - 6\ell - 4n -1 $,
$S$ reaches the maximum when $c = 3\Delta - 6\ell - 4n -1$.
In this case,
\begin{equation*}
\begin{aligned}
	S & \le \frac{5 \Delta^{2}}{2} + \Delta \left(- 8 \ell - 8 n - \frac{1}{2}\right) + 10 \ell^{2} + 16 \ell n + 3 \ell + 8 n^{2} + 2 n + \frac{1}{4}.
\end{aligned}
\end{equation*}
Since the coefficient of $\Delta^2$ is positive, $S$ achieves the maximum when $\Delta = 2n$ or $\Delta = (4n+6\ell+1)/3$.
When $\Delta = (4n+6\ell+1)/3$, $S < 2n^2 +2n$.
When $\Delta = 2n$,
\begin{equation*}
\begin{aligned}
	S \le 2 n^{2} + n + \frac{1}{4} + 10 \ell^{2} + 3 \ell<2n^2 + 2n,
\end{aligned}
\end{equation*}
 when $n$ is large enough. Thus we have a contradiction  with (\ref{eq: sum of degrees lower bound}) in both cases.

If $3\Delta - 6\ell - 4n -1 < 0$, i.e. $\Delta>(6\ell+4n+1)$/3, by $(4n-\Delta-c-2\ell)/2 \le n$, we have $c \ge \frac{2}{3}n-4\ell-1$.
Thus, $S$ reaches the maximum when $c = \frac{2}{3}n-4\ell-1$.

\begin{equation*}
\begin{aligned}
	S & \le \frac{\Delta^{2}}{4} + \Delta \left(\ell - 2 n + 1\right) + \ell^{2} + 4 \ell n + 4 n^{2}-\frac{1}{4}\left(\frac{2}{3}n-4\ell-1\right)^2\\
	&+\left(\frac{2}{3}n-4\ell-1\right)\left(\frac{3\Delta}{2}-3\ell-2n-\frac{1}{2}\right).
\end{aligned}
\end{equation*}
Since the coefficient of $\Delta^2$ is positive, $S$ achieves the maximum when $\Delta = (6\ell + 4n +1)/3$ or $\Delta=n+\sqrt{n}$.
% When $\Delta = (6\ell + 4n +1)/3$, we have
% $$S\le \frac{5}{3}n^2+7\ell n< 2n^2 + 2n$$
%  when $n$ is large enough, a contradiction  with (\ref{eq: sum of degrees lower bound}).
% When $\Delta = n$, we have
% $$S\leq 65/36 n^2+10\ell n<2n^2+2n$$
% when $n$ is large enough, a contradiction  with (\ref{eq: sum of degrees lower bound}).
In both cases, we have $S < 2n^2 + 2n$ when $n$ is large enough, a contradiction  with (\ref{eq: sum of degrees lower bound}).

\textbf{Subcase 2:} $(4n-\Delta -c- 2\ell)/2 > n$.
 Then, $S$ reaches the maximum when $a= n$. By (\ref{eq: count S}), taking $c = \Delta -n-\frac{1}{2}-4\ell$, we have
\begin{equation*}
\begin{aligned}
	S & \le  -\frac{1}{2}c^2+\left(\Delta -n-\frac{1}{2}-4\ell\right)c+3n^2-\Delta n+8\ell n+\Delta-2\ell\\
	&\le \frac{1}{2}\left(\Delta -n-\frac{1}{2}-4\ell\right)^2+3n^2-\Delta n+8\ell n+\Delta-2\ell.
\end{aligned}
\end{equation*}
Since the coefficient of $\Delta^2$ is positive, $S$ achieves the maximum when $\Delta=2n$ or $\Delta=n+\sqrt{n}$.
One can verify that in both cases, we have $S < 2n^2 + 2n$ when $n$ is large enough, a contradiction  with (\ref{eq: sum of degrees lower bound}).
% When $\Delta=2n$, we have $S\le \frac{3}{2}n^2+10\ell n< 2n^2+2n$ when $n$ is large enough, a contradiction    with (\ref{eq: sum of degrees lower bound}).

% When $\Delta=n+\sqrt{n}$, we have $S\le \frac{1}{2}n^{1.2}+2n^2-n^{1.6}+10\ell n< 2n^2+2n$ when $n$ is large enough, a contradiction    with (\ref{eq: sum of degrees lower bound}).

 %In these two subcases, we have a contradiction  with (\ref{eq: sum of degrees lower bound}).

\textbf{Case 2:} $a \ge n$.

In this case, $b\le n$.
Let $Y_0= \{ v \in Y \mid d(v) \ge a + 1\}.$
Then for each vertex $v \in Y_0$, it corresponds to at least $d(v)-a-1$ edges in $E_B$.
Since $\beta \le n \le a$, the vertices in $Y_0$ have different degrees by the definition of $\beta$.
So there are at least $0 + 1 + \ldots + y_0 = \frac{y_0(y_0+1)}{2}$ edges in $E_B$ corresponding to the vertices in $Y_0$.
Therefore, we have $y_0 \le \sqrt{4\ell b}$.

Now we count the sum of degrees of vertices in $G$.
\begin{enumerate}
	\item The sum of degrees of vertices in $V(C) \cup \{u\}$ is at most $\frac{(c+1)(2\Delta - c)}{2}$.
	\item The sum of degrees of vertices in $A$ is at most $a(2n-\Delta + 2\ell)$.
	\item The sum of degrees of vertices in $Z$ is at most $2n |Z| \le 8\ell b^{0.1}n$ by $|Z| \le 4\ell b^{0.1}$.
	\item The sum of degrees of vertices in $Y_0$ is at most $ay_0 + 4\ell b$.
	\item The sum of degrees of vertices in $Y \setminus Y_0$ is at most $a + (a-1) + \ldots + \beta + \ldots + \beta$, which achieves the maximum when $\beta = n$.
	It is at most
	\begin{equation*}
		\left\{
	\begin{array}{ll}
		\frac{(a + n+1)(a-n)}{2} + n(y - y_0 - (a-n)) & \text{if } y-y_0 \ge a-n,\\
		\frac{(a + a - (y-y_0)+1)(y-y_0)}{2} & \text{if } y-y_0 < a-n.\\
	\end{array}\right.
	\end{equation*}
	The second formula is always no larger than the first one, so we only use the first one to get an upper bound of $S$.
\end{enumerate}

\begin{equation*}
\begin{aligned}
	S & \le \frac{(c+1)(2\Delta - c)}{2} + a(2n-\Delta + 2\ell) + 8\ell b^{0.1}n + 4 \ell b \\
	& \quad + a y_0 + \frac{(a+n+1)(a-n)}{2} + n (y-y_0-(a-n)).
\end{aligned}
\end{equation*}

The coefficient of $y_0$ is $(a-n) \ge 0$, so we can replace $y_0$ by ${(4\ell b)^{1/2}}$ to get an upper bound of $S$.
Then the terms about $b$ is $4\ell b + (a-n)(4\ell b)^{1/2} + 8\ell b^{0.1}n$.
Since $b \le n$, we can replace $b$ by $n$ to get an upper bound of $S$.
Since $y=b-z\le b$, we replace $y$ by $b = 2n-c-a$.

Then we have
\begin{equation*}
\begin{aligned}
	S & \le \frac{(c+1)(2\Delta - c)}{2} + a(2n-\Delta + 2\ell) + 8\ell n^{1.1}  \\
	& \quad + \frac{(a+n+1)(a-n)}{2} + n (3n-c-2a) + (a-n)(4\ell n)^{1/2}.
\end{aligned}
\end{equation*}

For the lower order terms, give a rough upper bound~($a -n \le n^{0.6}$):
\begin{equation*}
\begin{aligned}
	S & \le \frac{1}{2}c(2\Delta - c) + \frac{1}{2}(a+n+1)(a-n) + n (3n-c-2a) + a(2n-\Delta) + O(n^{1.1})\\
	& =  \frac{a^{2}}{2} - \frac{c^{2}}{2} - \Delta a + c(\Delta  - n) + \frac{5 n^{2}}{2} + O(n^{1.1})\\
	% & \le n(a-c) - \Delta a + c(\Delta  - n) + \frac{5 n^{2}}{2} + O(n^{1.1})\\
	& \le \frac{a^{2}}{2} - \Delta a + \frac{5 n^{2}}{2} + O(n^{1.1}).\\
\end{aligned}
\end{equation*}

When $a \in [n, n + n^{0.6}]$, it achieves the maximum on the endpoints.
In both cases, $S \le 2n^2 - n^{1.5} +O(n^{1.1})$.
Then we have $S \le 2n^2 + n$ when $n$ is large enough, a contradiction  with (\ref{eq: sum of degrees lower bound}).

Thus, we have finished the proof of the theorem.
Moreover, through the proof, there are many inequalities needs $n$ to be large enough.
One can easily check that these inequalities are satisfied when $n$ is larger than a polynomial of $\ell$.

\section{Conclusion}\label{Conclusion}

Chen and Ma \cite{CHEN20261} also obtained a similar result for graphs with even number of vertices.
\begin{theorem}[Chen and Ma~\cite{CHEN20261}]
	There exists an integer $n_0 >0$ such that the following holds for all $n \ge n_0$.
	The unique $2n$-vertex graph with at least $n^2-1$ edges, that does not contain two vertices of the same degree joined by a path of length three, is the complete bipartite graph $K_{n-1,n+1}$.
\end{theorem}

Liu and Zeng~\cite{2025arXiv250500523L} improved the bound to $n \ge 3$.
And later, Liu and Zeng also extended the problem on paths of length five to graphs with even number of vertices.

\begin{theorem}[Liu and Zeng~\cite{2026arXiv260411664L}]
	Let $n \ge 13$ be an integer. The unique $2n$-vertex graph with at least $n^2-1$ edges, that does not contain two vertices of the same degree joined by a path of length five, is the complete bipartite graph $K_{n-1,n+1}$.
\end{theorem}

We believe that our methods can be extended for graphs with even number of vertices as well for $\ell \ge 3$.
For the sake of simplicity, we leave it as a future work.

% -------------------------------------------------------------
\section*{Acknowledgement}
% M. Lu is supported by the National Natural Science Foundation of China (Grant 12171272 \& 12161141003).
This research is supported by the National Natural Science Foundation of China  (Grant 12571372).

\section*{Declaration of competing interest}
The authors declare that they have no known competing financial interests or personal relationships that could have appeared to influence the work reported in this paper.

\section*{Data availability}
No data was used for the research described in the article.

% ------------ bib part ---------------------


\begin{thebibliography}{1}

	\bibitem{CHEN20261}
	K.~Chen and J.~Ma.
	\newblock A problem of Erd{\H{o}}s and Hajnal on paths with equal-degree endpoints.
	\newblock \emph{Journal of Combinatorial Theory, Series B}, 179:1--18, 2026.
	
	\bibitem{erdos1991problems}
	P.~Erd\H{o}s.
	\newblock Problems and results in combinatorial analysis and combinatorial number theory.
	\newblock \emph{Graph theory, combinatorics, and applications}, 1:397--406, 1991.
	
	\bibitem{erdos1959maximal}
	P.~Erd\H{o}s and T.~Gallai.
	\newblock On maximal paths and circuits of graphs.
	\newblock \emph{Acta Math. Acad. Sci. Hungar.}, 10(3):337--356, 1959.
	
	\bibitem{2025arXiv250500523L}
	Z.~{Liu} and Q.~{Zeng}.
	\newblock {A complement of the Erd{\H{o}}s-Hajnal problem on paths with equal-degree endpoints}.
	\newblock \emph{arXiv e-prints}, arXiv:2505.00523, May 2025.
	
	\bibitem{2026arXiv260411664L}
	Z.~{Liu} and Q.~{Zeng}.
	\newblock {Paths of length five with equal-degree endpoints}.
	\newblock \emph{arXiv e-prints}, arXiv:2604.11664, 2026.
	
	\end{thebibliography}
\end{document}